\documentclass{amsart}
\usepackage{amssymb}
\usepackage{mathrsfs}
\usepackage{graphicx}
\usepackage{comment}

\newcommand{\be}{\begin{equation}}
\newcommand{\ee}{\end{equation}}
\newcommand{\ba}{\begin{align}}
\newcommand{\ea}{\end{align}}

\newtheorem{theorem}{Theorem}[section]
\newtheorem{lemma}[theorem]{Lemma}

\newtheorem*{proposition*}{Proposition}
\newtheorem{corollary}[theorem]{Corollary}
\newtheorem*{theorem*}{Theorem}
\newtheorem*{corollary*}{Corollary}
\newtheorem*{cor*}{Corollary}	
	
\theoremstyle{definition}

\newtheorem{conjecture}[theorem]{Conjecture}	

{\begin{list}{}{%
\settowidth{\labelwidth}{\textsf{{\it #1.}}}%
\setlength{\labelsep}{4mm}%
\setlength{\leftmargin}{\labelwidth}%
\addtolength{\leftmargin}{\labelsep}%
}}%
{\end{list}}

\def\beq{\begin{equation}}\def\enq{\end{equation}}

{\begin{list}{}{%
\settowidth{\labelwidth}{\textsf{{\it #1.}}}%
\setlength{\labelsep}{2mm}%
\setlength{\leftmargin}{\labelwidth}%
\addtolength{\leftmargin}{\labelsep}%
\addtolength{\leftmargin}{4mm}%
\setlength{\itemsep}{6pt}%
\setlength{\listparindent}{0pt}%
\setlength{\topsep}{3pt}%
}}%

\usepackage{color}
\usepackage[normalem]{ulem} 
\usepackage{soul}

\usepackage{xcolor}

\begin{document}

\title[Extending recent work of Nath, Saikia, and Sarma on $k$-tuple $\ell$-regular Partitions]{Extending recent work of Nath, Saikia, and Sarma on $k$-tuple $\ell$-regular Partitions }

\author[B. Paudel]{Bishnu Paudel}

\author[J. Sellers]{James A. Sellers}

\author[H. Wang]{Haiyang Wang}

\address{Mathematics and Statistics Department\\
         University of Minnesota Duluth\\
         Duluth, MN 55812, USA}
\email{bpaudel@d.umn.edu, jsellers@d.umn.edu, wan02600@d.umn.edu}

\begin{abstract}
Let $T_{\ell,k}(n)$ denote the number of $\ell$-regular $k$-tuple partitions of $n$. In a recent work, Nath, Saikia, and Sarma derived several families of congruences for $T_{\ell,k}(n)$, with particular emphasis on the cases $T_{2,3}(n)$ and $T_{4,3}(n)$. In the concluding remarks of their paper, they conjectured that $T_{2,3}(n)$ satisfies an infinite set of congruences modulo 6. In this paper, we confirm their conjecture by proving a much more general result using elementary $q$-series techniques. We also present new families of congruences  satisfied by $T_{\ell,k}(n)$.
\end{abstract}

\maketitle

\section{Introduction}

A \textit{partition} of a non-negative integer $n$ is a non-increasing sequence of positive integers, called parts, whose sum is $n$. By convention, zero has only one partition, namely, the empty sequence. Let $p(n)$ denote the number of partitions of $n$. Ramanujan \cite{Ramanujan} famously established the celebrated congruences
\begin{align*}
	p(5n+4)&\equiv 0\pmod 5,\\
	p(7n+5)&\equiv 0\pmod 7,\\
	p(11n+6)&\equiv 0\pmod {11}.
\end{align*}

Since Ramanujan’s pioneering work, mathematicians have investigated further congruences for $p(n)$ and the arithmetic properties of its various generalizations. In this paper, we consider the following generalization. If
\[
n_1, n_2, \dots, n_k \ge 0 \quad \text{with} \quad n_1 + n_2 + \cdots + n_k = n,
\]
and if $\lambda_i$ is a partition of $n_i$ for each $i=1,\dots,k$, then the sequence
\[
(\lambda_1, \lambda_2, \dots, \lambda_k)
\]
is called a \textit{$k$-tuple} partition of $n$. For instance, if $\lambda_1 = (3,2,1)$ and $\lambda_2 = (7,6,2)$, then $(\lambda_1, \lambda_2)$ forms a 2-tuple partition of $21$.  

A partition is called \textit{$\ell$-regular} if none of its parts is divisible by $\ell$. Correspondingly, a $k$-tuple partition is said to be \textit{$\ell$-regular} if each $\lambda_i$ is an $\ell$-regular partition of $n_i$ for $1 \le i \le k$. We denote the number of $k$-tuple $\ell$-regular partitions of $n$ by $T_{\ell,k}(n)$ and, in particular, define
\[
T_\ell(n) := T_{\ell,3}(n).
\]

The theory of $\ell$-regular partitions has been extensively developed in the literature (see, for example, \cite{Furcy-Penniston}, \cite{Gordon-Ono}, and \cite{Hirschhorn-Sellers}). More recently, $\ell$-regular $k$-tuple partitions have attracted significant interest, and various divisibility properties of $T_{\ell,k}(n)$ have been explored.  For instance, the case $(\ell,k)=(3,3)$ was studied by Adiga and Dasappa \cite{Adiga-Dasappa}, da Silva and Sellers \cite{daSilva-Sellers}, and Gireesh and Mahadeva Naika \cite{Gireesh-Mahadeva_Naika}; the cases $(\ell,k)=(3,9)$ and $(3,27)$ by Baruah and Das \cite{Baruah-Das}; the case $(\ell,k)=(3,6)$ by Murugan and Fathima \cite{Murugan-Fathima}; and both $(\ell,k)=(2,3)$ and $(3,3)$ by Nadji and Ahmia \cite{Nadji-Ahmia}. Additionally, Rahman and Saikia \cite{Rahman-Saikia} examined the cases $(\ell,k)=(5,3)$ and $(5,5)$, while Vidya \cite{Vidya} considered cases with $k=3$ and $\ell\in \{2,4,10,20\}$.

In a recent work \cite{Nath-Saikia-Sarma}, Nath, Saikia, and Sarma analyzed the cases $(\ell,k)=(2,3)$ and $(4,3)$, and in some instances for general $(\ell,k)$. 
In particular, they proved  
\cite[Theorem 1.3]{Nath-Saikia-Sarma} that for $n\geq 0$ and $\alpha\geq 0$, we have
 \begin{align}     
\label{TT29n24}T_{2}\left(3^{4\alpha+2}n+\sum_{i=0}^{2\alpha}3^{2i}+3^{4\alpha+1}\right)&\equiv 0\pmod{24}, \\
\label{TTT29n24}T_{2}\left(3^{4\alpha+2}n+\sum_{i=0}^{2\alpha}3^{2i}+2\cdot 3^{4\alpha+1}\right)&\equiv 0\pmod{24}.
 \end{align}
They also proved \cite[Theorem 1.6]{Nath-Saikia-Sarma} that if $p\equiv 5 \text{ or } 7\pmod{8}$ and $n,\alpha\geq0$ with $p\nmid n$, then
\begin{equation}\label{1.6mod6}
T_2\left(9p^{2\alpha+1}n+\frac{9p^{2\alpha+2}-1}{8}\right)\equiv0\pmod{6}.\end{equation}
Moreover, their analysis led them to propose the following conjecture: 

\begin{conjecture}[Nath, Saikia, Sarma \cite{Nath-Saikia-Sarma}]\label{conj.nss}
	Let $p\geq 5$ be a prime with $\left( \dfrac{-2}{p} \right)_{L}=-1$ and, let $t$ be a positive integer with $(t,6) = 1$ and $p\mid t$. Then for all $n\geq 0$ and $1\leq j \leq p-1$, we have 
	\begin{equation*}
		T_2\left( 9\cdot t^2n + \frac{9\cdot t^2j}{p} + \frac{57\cdot t^2-1}{8} \right) \equiv 0 \pmod{6}.
	\end{equation*}
\end{conjecture}

Motivated by this conjecture, we establish the following theorem as a stronger version of it.

\begin{theorem} \label{thm.nssconj}
    Let $t$ be a positive integers with $\gcd(t,6)=1$. Then, for $n\geq 0$ and $N=33 \text{ or } 57$, we have 
    \begin{equation}\label{T29n24}
        T_2\left(9n+\frac{Nt^2-1}{8}\right)\equiv 0\pmod{24}.
        \end{equation}
\end{theorem}
\begin{corollary}
    Conjecture \ref{conj.nss} is true. 
\end{corollary}
\begin{proof}
    Using $N=57$ and replacing $n$ by $t^2n+\cfrac{t^2j}{p}$ in \eqref{T29n24} completes the proof.
\end{proof}

With the goal of extending such congruences even further, we prove the following infinite family of congruences modulo 8 in extremely elementary fashion. 

\begin{theorem}\label{thm.p357mod8}
	Let $p\equiv3, 5 \text{ or } 7\mod 8$ be a prime. Then, for $n,\alpha\geq 0$ with $p\nmid n$, we have $$T_2\left(p^{2\alpha+1}n+\frac{p^{2\alpha+2}-1}{8}\right)\equiv0\pmod{8}.$$
In addition, if $3\nmid n$, then 
\begin{equation}\label{mod24}
T_2\left(p^{2\alpha+1}n+\frac{p^{2\alpha+2}-1}{8}\right)\equiv0\pmod{24}.\end{equation}
\end{theorem}

Lastly, we note that \eqref{1.6mod6} holds modulo 24, and it also holds for primes $p\equiv3\pmod{8}$.

\begin{corollary}\label{cor.9pmod24}
    Let $p\equiv 3, 5 \text{ or } 7\mod 8$ be a prime and $p\ne 3$. Then, for $n,\alpha\geq 0$ with $p\nmid n$, we have $$T_2\left(9p^{2\alpha+1}n+\frac{9p^{2\alpha+2}-1}{8}\right)\equiv0\pmod{24}.$$
\end{corollary}

\begin{proof}
Replacing $n$ by $9n+p$ in \eqref{mod24} gives the result.
\end{proof}

After collecting a number of necessary mathematical tools in Section \ref{sec:preliminaries}, we provide elementary proofs of Theorem \ref{thm.nssconj} and Theorem \ref{thm.p357mod8} in Section \ref{sec:ourproofs}, and we share several closing comments in Section \ref{sec:closingthoughts}.

\section{Preliminaries}\label{sec:preliminaries}

Our proofs of the aforementioned theorems are entirely elementary, relying solely on generating function manipulations and classical $q$-series results. In this section, we present several elementary facts, obtained from elementary $q$-series analysis, that will be utilized in the course of our proofs below.

We recall the $q$-Pochhammer symbol defined by
\[
(a;q)_\infty := \prod_{i=0}^{\infty} (1 - aq^i)
\]
and denote
\[
f_k := (q^k;q^k)_\infty.
\]
With this notation, it is clear that the generating function for $T_{\ell,k}(n)$ is given by
\[
\sum_{n \ge 0} T_{\ell,k}(n)q^n = \frac{f_\ell^k}{f_1^k},
\]
and, in particular, the generating function for $T_\ell(n)$ is
\[
\sum_{n \ge 0} T_\ell(n)q^n = \frac{f_\ell^3}{f_1^3}.
\]

We now collect the results which will be necessary in our work below.

\begin{lemma}  We have
    \begin{equation}\label{f1}
        f_1=\sum_{m\in\mathbb{Z}}(-1)^mq^{m(3m-1)/2}.
    \end{equation}
\end{lemma}
\begin{proof}
    A proof of this identity can be found in \cite[Section 1.6]{Hirschhorn}.
\end{proof}
\begin{lemma}   We have
\begin{equation}\label{f13}
    f_1^3=\sum_{m\ge 0}(-1)^m(2m+1)q^{m(m+1)/2}.
\end{equation}
\end{lemma}
\begin{proof}
    This identity can be found in \cite[(1.7.1)]{Hirschhorn}.
\end{proof}

\begin{lemma}  We have
  \begin{equation}\label{f15/f22}
        \frac{f_1^5}{f_2^2}=\sum_{m\in\mathbb{Z}}(6m+1)q^{m(3m+1)/2}.
    \end{equation}  
\end{lemma}
\begin{proof}
    This identity appears in \cite[Equation (10.7.3)]{Hirschhorn}.
\end{proof}
\begin{lemma}  We have
\begin{equation}\label{f1-q}
    (-q;-q)_\infty=\frac{f_2^3}{f_1 f_4}.
\end{equation}
\end{lemma}
\begin{proof} Note that 
\begin{align*}
\frac{f_2^3}{f_1f_4}
&=\frac{(q^2;q^2)_\infty^3}{(q;q)_\infty(q^4;q^4)_\infty} \\
&=\frac{(q^2;q^2)_\infty^2}{(q;q^2)_\infty(q^4;q^4)_\infty} \\
&=\frac{(q^2;q^2)_\infty(q^2;q^4)_\infty}{(q;q^2)_\infty}\\
&=(q^2;q^2)_\infty(-q;q^2)_\infty\\
&= (-q;-q)_\infty.
\end{align*}
\end{proof}
\begin{lemma}\label{lemma:f12overf21}  We have
\begin{equation}\label{f12/f2}
\frac{f_1^2}{f_2}=\sum_{n\in\mathbb{Z}}(-1)^nq^{n^2}=1+2\sum\limits_{n\ge 1}(-1)^nq^{n^2}.\end{equation}
\end{lemma}
\begin{proof}
    See \cite[(1.5.8)]{Hirschhorn} for a proof of this result.
\end{proof}

\begin{corollary}
      We have
\begin{equation}\label{f22/f4}
    \frac{f_2^2}{f_4}=1+2\sum\limits_{n\ge 1}(-1)^nq^{2n^2}.
\end{equation}
\end{corollary}
\begin{proof}

This follows from Lemma \ref{lemma:f12overf21} by replacing $q$ by $q^2$ everywhere in (\ref{f12/f2}).
\end{proof}
We now provide the congruence--related tools necessary to complete our proofs in the next section.  We begin with an extremely well--known result which, in essence, follows from the Binomial Theorem. 
\begin{lemma}\label{fdream}
    For a prime $p$ and positive integers $k$ and $l$,
    \begin{equation}\label{modp^k}
        f_l^{p^k}\equiv f_{lp}^{p^{k-1}}\pmod{p^k}.
    \end{equation}
\end{lemma}
\begin{proof}
    See \cite[Lemma 3]{daSilva-Sellers_20} for a proof.
\end{proof}

In \cite[Theorem 1.1]{Nath-Saikia-Sarma}, Nath, Saikia, and Sarma use Lemma \ref{fdream} to prove that, if $p$ is a prime and $1\leq r\leq p-1$, then 
\begin{equation}\label{modp}
T_{\ell,p}(pn+r)\equiv0\pmod{p}.\end{equation}

Note that Lemma \ref{fdream} can be generalized in a natural way in order to obtain the following:  

\begin{lemma}\label{cor.modp^k}
For a prime $p$ and positive integers $k,l,\text{ and }s$ with $k-s\geq0$,
\begin{equation}
        f_l^{p^km}\equiv f_{lp^{s}}^{p^{k-s}m}\pmod{p^{k-s+1}}.
    \end{equation}
\end{lemma}
\begin{proof}
    The congruence follows by applying \eqref{modp^k} $s$ times.
\end{proof}

With Lemma \ref{cor.modp^k} in hand, we can introduce a generalized version of the congruence in (\ref{modp}), the proof of which is almost immediate. 

\begin{theorem}\label{thm.gen1.1}
Let $p$ be a prime and $m, \ell>0$ be integers. For $\alpha, s>0$ satisfying $\alpha-s\ge 0$ and $1\leq r\leq p^s-1$,
\[T_{\ell,p^\alpha m}(p^s n+r)\equiv 0 \pmod{p^{\alpha-s+1}}.\]
\end{theorem}
\begin{proof}
We use Lemma \ref{cor.modp^k} to obtain 
    \begin{align*}
            \sum_{n\geq0} T_{\ell,p^\alpha m}(n)p^n=\frac{f_\ell^{p^\alpha m}}{f_1^{p^\alpha m}}\equiv \frac{f_{p^s\ell}^{p^{\alpha-s} m}}{f_{p^s}^{p^{\alpha-s} m}}\pmod{p^{\alpha-s+1}}.
    \end{align*}
Notice that the right-hand side is a function of $q^{p^s}$. Thus, the coefficients of $q^{p^sn+r}$ with $1\leq r\leq p^s-1$ equal zero. This completes the proof.
\end{proof}

We close this section by proving a parity characterization for the particular function $T_2(n)$.
\begin{lemma} \label{mod2}
For all $n\geq 0$, 
$$
T_2(n) \equiv \begin{cases}
			1\pmod{2}, & \text{if $n = m(m+1)/2$ for some $m$}\\
            0\pmod{2}, & \text{otherwise}.
		 \end{cases}
$$
\end{lemma}
\begin{proof}
We see that 
\begin{align*}
\sum_{n\geq0}T_2(n)q^n
&= \frac{f_2^3}{f_1^3} \\
&\equiv \frac{f_1^6}{f_1^3} \pmod{2} \text{\ \ thanks to  \eqref{modp^k}} \\
&= f_1^3.
\end{align*}
Thanks to  \eqref{f13}, the result follows.  
\end{proof}

With the above tools in hand, we now proceed to prove Theorem \ref{thm.nssconj} and Theorem \ref{thm.p357mod8}. 

\section{Proofs of Our Results}\label{sec:ourproofs}

We begin by sharing our proof of Theorem \ref{thm.nssconj} which generalizes the original conjecture of Nath, Saikia, and Sarma \cite{Nath-Saikia-Sarma}.  

\begin{proof}[Proof (of Theorem \ref{thm.nssconj})]
Using $\alpha=0$ in \eqref{TT29n24} and \eqref{TTT29n24} gives
\begin{align*}
    T_2(9n+4)&\equiv 0\pmod{24},\\
    T_2(9n+7)&\equiv0\pmod{24},
\end{align*}
respectively. Thus, writing $$k=9n+\frac{Nt^2-1}{8},$$
it suffices to show that $k\equiv4 \text{ or } 7\pmod 9$. Since $\gcd(t,6)=1$, $t$ must have one of the forms $6m+1 \text{ or } 6m+5$. Therefore, 
\[
Nt^2 \equiv\begin{cases}
    6 \pmod{9} & \text{ if } N=33,\\
    3\pmod{9} & \text{ if } N=57.
\end{cases}
\]
Observing that $8^{-1}\equiv-1\pmod{9}$, we have  
\[
k\equiv \begin{cases}
    4 \pmod{9} & \text{ if } N=33,\\
    7\pmod{9} & \text{ if } N=57.
\end{cases}
\]\end{proof}


Next, we provide an elementary proof of Theorem \ref{thm.p357mod8}.  
\begin{proof}[Proof (of Theorem \ref{thm.p357mod8})]
 Thanks to \eqref{modp^k}, we have 
\begin{equation}\label{mod8}
    \sum_{n\geq0}T_2(n)q^n=\frac{f_2^3}{f_1^3}=\frac{f_2^4}{f_1^3f_2}\equiv\frac{f_1^5}{f_2^2}f_2 \pmod8.
\end{equation}
Using \eqref{f1} and \eqref{f15/f22} in \eqref{mod8}, we obtain 
\begin{align*}
    \sum_{n\geq0}T_2(n)q^n\equiv \sum_{m\in\mathbb{Z}}\sum_{k\in\mathbb{Z}}(-)^k(6m+1)q^{m(3m+1)/2+k(3k-1)} \pmod 8.
\end{align*}
We now need to check whether $l:=p^{2\alpha+1}n+\cfrac{p^{2\alpha+2}-1}{8}$ can be represented as $\cfrac{m(3m+1)}{2}+k(3k-1)$. Equivalently, we check whether $$24l+3=(6m+1)^2+2(6k-1)^2.$$ 
Let $\nu_p(N)$ be the highest power of $p$ dividing $N$. We first consider primes $p\equiv5\text{ or }7\mod 8$. Then, we have $\left(\cfrac{-2}{p}\right)=-1$. So, if $N=x^2+2y^2$, then $2|\nu_p(N)$. Also, for $p\nmid n$, we have $\nu_p(24l+3)=2\alpha+1$. Thus, $24l+3$ cannot be of the form $x^2+2y^2$. Hence, for $p\equiv5 \text{ or } 7 \mod 8$,
$$T_2\left(p^{2\alpha+1}n+\frac{p^{2\alpha+2}-1}{8}\right)\equiv 0\pmod{8}.$$

The above proof only deals with the cases $p\equiv5 \text{ or } 7 \mod 8$, which still leaves us with the case $p\equiv3 \mod 8$.  This requires a different idea.  In this case, we replace 
$q$ by $-q$ and then, using \eqref{f1-q} and \eqref{modp^k}, we obtain 
\begin{equation} \label{mod82}
    \sum_{n\geq0}T_2(n)(-q)^n =f_2^3\left(\frac{f_1f_4}{f_2^3}\right)^3=\frac{f_1^3f_4^3f_2^2}{f_2^8}\equiv \frac{f_1^3f_4^3f_2^2}{f_4^4}= f_1^3\frac{f_2^2}{f_4} \pmod{8}.
\end{equation}
Substituting \eqref{f13} and \eqref{f22/f4} in \eqref{mod82} gives
\begin{align*}
   \sum_{n\geq0}T_2(n)(-q)^n\equiv \sum_{m\in\mathbb{Z}}(-1)^m(2m+1)q^{m(m+1)/2}\left(1+2\sum_{k\geq 1}(-1)^kq^{2k^2}\right) \pmod{8}.
\end{align*}
Now we need to check whether we have 
$$l=\frac{m(m+1)}{2}+2k^2,$$
that is, $$8l+1=(2m+1)^2+(4k)^2.$$
For primes $p\equiv 3\mod 4$, we note that $\left(\cfrac{-1}{p}\right)=-1$. So, for any $N=x^2+y^2$, we have $2|\nu_p(N)$. However, for $p\nmid n$,  $\nu_p(8l+1)=2\alpha+1$. This completes the proof of the congruences modulo 8.

For the second part, it suffices to show that $T_2(l)\equiv 0\pmod{3}$. For this, we check whether $3|l$ (in order to apply \eqref{modp}). Since $8^{-1}\equiv-1 \pmod 3$, we have 
\[l=p^{2\alpha+1}n+\cfrac{p^{2\alpha+2}-1}{8}\equiv
\begin{cases}n \pmod{3} &\text{ if } p\equiv1\pmod{3},\\
-n \pmod{3} & \text{ if } p\equiv -1\pmod {3},\\
1 \pmod{3} & \text{ if } p=3.
\end{cases}
\]
Since $3\nmid n$, we have $3\nmid l$ and, by \eqref{modp}, $T_2(l)\equiv0\pmod 3$.

\end{proof}

\section{Closing Thoughts}\label{sec:closingthoughts}

We conclude by mentioning some additional observations. 
In the case when $\alpha = 0$ and $n$ is replaced by $pn+s$ for $1\leq s \leq p-1$, Theorem \ref{thm.p357mod8} shows that
\[
T_2(p^2n+r)\equiv 0\pmod{8},
\]
when
$p\equiv 3, 5, \text{ or } 7\pmod{8}$ is prime, and $r=\displaystyle{ps+\frac{p^2-1}{8}}$. 
Our computations further indicate that $T_2(n)$ is divisible by $32$ in the following cases:
\[
\begin{aligned}
	T_2(25n+8)&\equiv 0 \pmod{32},\\[1mm]
	T_2(25n+13)&\equiv 0 \pmod{32},\\[1mm]
	T_2(25n+18)&\equiv 0 \pmod{32},\\[1mm]
	T_2(25n+23)&\equiv 0 \pmod{32},\\[1mm]
	T_2(49n+13)&\equiv 0 \pmod{32},\\[1mm]
	T_2(49n+20)&\equiv 0 \pmod{32},\\[1mm]
	T_2(49n+27)&\equiv 0 \pmod{32},\\[1mm]
	T_2(49n+34)&\equiv 0 \pmod{32},\\[1mm]
	T_2(49n+41)&\equiv 0 \pmod{32},\\[1mm]
	T_2(49n+48)&\equiv 0 \pmod{32}.
\end{aligned}
\]
It would be interesting to determine whether these congruences hold for all $n\geq 0$, and whether they extend to $T_2(p^2n+r)$ for primes $p>7$. Moreover, our computations suggest that, with only a few exceptions, the congruence
\[
T_2(7n+6)\equiv 0\pmod{2^9}
\]
appears to hold. We leave these questions for the interested reader.

\bibliographystyle{plain}
\nocite{*}
\bibliography{k_tuples_l_regular.bib}


\end{document}